\documentclass[11pt,reqno]{amsart}
\usepackage[toc,page]{appendix}
\usepackage{amsmath,amsfonts,amssymb,amsthm,mathtools}
\usepackage[vcentermath]{youngtab}
\usepackage[all]{xy}
\usepackage{csquotes}
\usepackage[usenames,dvipsnames]{xcolor}
\usepackage[demo]{graphicx}
\usepackage{subcaption}
\usepackage{enumerate}
\usepackage{pstricks}
\usepackage{filecontents}
\usepackage{graphicx}
\usepackage{cite}
\usepackage{setspace}
\usepackage{tabu}
\usepackage{tikz}
\usetikzlibrary{calc}
\usetikzlibrary{decorations.markings}
\usetikzlibrary{patterns}
\usepackage[margin=1in]{geometry}
\usepackage{siunitx} %
\usepackage{hyperref}
\hypersetup{colorlinks,  citecolor=blue,  linkcolor=BrickRed, urlcolor=black}
\usepackage{pgf}
\usetikzlibrary{arrows, automata}
\usetikzlibrary{chains}
\usepackage{indentfirst}
\usepackage{comment}
\usepackage{todonotes}
\DeclareFontFamily{OT1}{rsfs}{}
\DeclareFontShape{OT1}{rsfs}{n}{it}{<-> rsfs10}{}
\DeclareMathAlphabet{\mathscr}{OT1}{rsfs}{n}{it}

\DeclareMathOperator{\rank}{rank}

\makeatletter
\renewcommand{\boxed}[1]{\text{\fboxsep=.2em\fbox{\m@th$\displaystyle#1$}}}
\makeatother

\newcommand{\con}{\textrm{-}}
\newcommand{\mcA}{\mathcal{A}}
\newcommand{\mcB}{\mathcal{B}}

\newcommand{\mbF}{\mathbb{F}}

\newcommand{\mbZ}{\mathbb{Z}}

\theoremstyle{plain}
\newtheorem{Theorem}{Theorem}[section]
\newtheorem{Proposition}[Theorem]{Proposition}
\newtheorem{Lemma}[Theorem]{Lemma}
\newtheorem{Corollary}[Theorem]{Corollary}

\newtheorem{Example}[Theorem]{Example}

\theoremstyle{definition}
\newtheorem{Definition}[Theorem]{Definition}

\theoremstyle{remark}

\usepackage[foot]{amsaddr}

\makeatletter
\renewcommand{\email}[2][]{%
  \ifx\emails\@empty\relax\else{\g@addto@macro\emails{,\space}}\fi%
  \@ifnotempty{#1}{\g@addto@macro\emails{\textrm{(#1)}\space}}%
  \g@addto@macro\emails{#2}%
}
\makeatother

\title{The title}%

\author{Nicholas Guo}
\email[N. Guo]{nicholas.guo.bowling@gmail.com}
\address[N. Guo]{Massachusetts Institute of Technology\\ Cambridge, MA, USA}

\author{Guangyi Yue}
\address[G. Yue]{Department of Mathematics\\ Massachusetts Institute of Technology\\ Cambridge, MA, 02139, USA}
\email[G. Yue]{gyyue@mit.edu}

\begin{document}

\setcounter{tocdepth}{4}

\title[Counting Independent Sets in Graphs of Hyperplane Arrangements]{Counting Independent Sets in Graphs of Hyperplane Arrangements}

\maketitle
\begin{abstract}
In this paper, we count the number of independent sets of a type of graph $G(\mathcal{A},q)$ associated to some hyperplane arrangement $\mathcal{A}$, which is a generalization of the construction of graphical arrangements. We show that when the parameters of $\mathcal{A}$ satisfy certain conditions, the number of independent sets of the disjoint union $G(\mcA,q_1)\cup\cdots\cup G(\mcA,q_s)$ depends only on the coefficients of $\mcA$ and the total number of vertices $\sum_i q_i$ when $q_i$'s are powers of large enough prime numbers. In addition it is independent of the coefficients as long as $\mcA$ is central and the coefficients are multiplicatively independent.
\end{abstract}

\section{Introduction}
A (real) hyperplane arrangement is a finite collection $\mathcal{A}$ of affine hyperplanes in $\mathbb{R}^n$ and the well-known Braid arrangement $\mathcal{B}_n$ consists of hyperplanes $x_i=x_j$ for $1\leq i<j\leq n$. Any sub-arrangement $\mcA$ of $\mathcal{B}_n$ is called a graphical arrangement because it can be naturally associated to a graph $G$ on $[n]$ where $ij\in E(G)$ iff $x_i=x_j\in\mcA$. The invariants of a graphical arrangement and its corresponding graph are closely related. For instance, the characteristic polynomial of a graphical arrangement $\mcA$ is the same as the chromatic polynomial of the corresponding $G$, see \cite{stanley1}. Whitney gave a classical formula in \cite{whitney1932logical} to calculate this polynomial:
$$\chi_\mcA(t)=\sum_{S\subseteq E(G)}(-1)^{\#S}q^{c(S)}$$
where $c(S)$ is the number of connected components of the spanning subgraph of $G$ with edge set $S$. Postnikov and Stanley generalized this result to arbitrary deformations of the Braid arrangement in \cite{postnikov2000deformations}. 

In this paper, rather than building a graph whose number of vertices is the dimension $n$ of the ambient space, we fix the number of vertices $q$ as an arbitrary power of a large enough prime number, but make the dimension $n$ to appear as $n$-independent sets of the corresponding graph. The construction is as follows: given parameters $A=\{(a_1,b_1),\ldots,(a_m,b_m)\}\subset\mbZ\times\mbZ$, let $$\mcA_n=\mcA_n(A)\coloneqq\mathcal{B}_n\cup\left\{x_i=a_kx_j+b_k\mid1\leq i\neq j\leq n, 1\leq k\leq m\right\}.$$
We define the associated graph $G(\mcA_n,q)$ to have vertex set $[q]$ and edges $ij\in E(G(\mcA_n,q))$ iff $i\equiv a_k j+b_k(\text{mod }q)$ for some $k$. We show in the paper that the number of $n$-element independent sets of $G(\mcA_n,q)$ is given by $\chi_{\mcA_n}(q)/n!$ when $q$ is power of a large enough prime number. Also we study in detail the following three cases:
\begin{enumerate}
\item $A_0=\{(1,1)\}$;
\item $A_1=\{(0,0),(a_1,0),\ldots,(a_m,0)\}$;
\item $A_2=\{(1,b_1),\ldots,(1,b_m)\}$.
\end{enumerate}
The first case $\mcA_n(A_0)$ is the Catalan arrangement with the corresponding graph being a cycle. The main result of the paper shows that when $A=A_1$ and $a_i$'s are multiplicatively independent, the number of $n$-element independent sets in the disjoint union $G(\mcA_n(A),q_1)\cup\cdots \cup G(\mcA_n(A),q_s)$ only depends on $m$, $n$ and $\sum_i(q_i-1)$ when $q_i$'s are powers of large enough prime numbers. In case of $A=A_2$, the number of  $n$-element independent sets in this disjoint union depends only on $n$, $b_i$'s and $\sum_iq_i$, for all $q_i$ large enough (not necessarily prime powers).

The rest of the paper is organized as follows. Section 2 is an overview of the basic properties of hyperplane arrangements and an application of the finite field method which relates characteristic polynomials of $\mathcal{A}_n(A)$ and independent sets of the corresponding graph $G(\mcA_n(A),q)$. Section 3 deals with the characteristic polynomial of $\mcA_n(A_1)$, and its properties when the parameters $a_i$'s are multiplicatively independent. Section 4 proves the main theorems regarding independent sets in the disjoint union of such graphs $G(\mcA_n(A),q)$ under different assumptions on $A$. And finally in Section 5, we go back to the characteristic polynomial of $\mcA_n(A_1)$ by relating it with the extended Catalan arrangements, and several special cases are explicitly computed.

\section{Preliminaries}



In this section, we give a brief overview of the basic properties of hyperplane arrangements following \cite{stanley1}.  A hyperplane arrangement is a finite collection $\mathcal{A}$ of affine hyperplanes in $\mathbb{R}^n$ and its intersection poset $L(\mathcal{A})$ consists of all intersections of subsets of the hyperplanes in $\mathcal{A}$ ordered by reverse inclusion with a unique minimal element $\hat{0}=\mathbb{R}^n$. We call $\mathcal{A}$ central if $\cap_{H\in\mathcal{A}}H\ne \emptyset$. Denote by $\textnormal{rank}(\mathcal{A})$ the rank of $\mathcal{A}$ which is the dimension of the space spanned by all the normals to the hyperplanes in $\mathcal{A}$.  The characteristic polynomial $\chi_{\mathcal{A}}$ is defined as $\chi_{\mathcal{A}}(t)= \sum_{x\in L(\mathcal{A})} \mu(\hat{0},x)t^{\dim(x)},$ where $\mu(\hat{0},x)$ is the M\"{o}bius function of $L(\mathcal{A})$. The characteristic polynomial can also be obtained by summing over all central subsets of $\mathcal{A}$ as follows.

\begin{Theorem}[\cite{orlik2013arrangements}]\label{Whitney}
Let $\mathcal{A}$ be an arrangement in $\mathbb{R}^n$. Then, $$\chi_{\mathcal{A}_n}(t)=\sum_{\substack{\mathcal{B}\subseteq\mathcal{A}\\\mathcal{B}\; \textnormal{central}}}(-1)^{\#\mathcal{B}}t^{n-\textnormal{rank}(\mathcal{B})}.$$
\end{Theorem}

We call an arrangement $\mcA$ essential if $\rank(\mcA)=n$. Denote by $r(\mathcal{A})$ the number of connected components of $\mathbb{R}^n\setminus\cup_{H\in\mathcal{A}}H$ and by $b(\mathcal{A})$ the number of relatively bounded regions \cite{stanley1}, which exactly means bounded when $\mcA$ is essential. We have the following theorem to calculate $r(\mcA)$ and $b(\mcA)$.

\begin{Theorem}[\cite{zaslavsky}]\label{Zaslavsky}
Given a hyperplane arrangement $\mathcal{A}$ in $\mathbb{R}^n$, we have 
\begin{align*}
r(\mathcal{A})=& (-1)^n\chi_{\mathcal{A}}(-1)\\
b(\mathcal{A})= & (-1)^{\textnormal{rank}(\mathcal{A})}\chi_{\mathcal{A}}(1).
\end{align*} 
\end{Theorem}

In this paper, we are dealing with arrangements $\mcA$ defined over rational numbers, and by multiplying each hyperplane equation by a suitable integer, we may assume $\mcA$ is defined over integers. Then we can take the coefficients modulo a prime number $p$ and get an arrangement $\mcA(q)$ defined over the finite filed $\mbF_q$, where $q = p^r$ for some $r$. The following finite field method for calculating the characteristic polynomial for such arrangements is crucial.

\begin{Theorem}[\cite{garsia2002saga}] 
\label{finite field}
Let $\mathcal{A}$ be an arrangement defined over integers and $q=p^r$ for $p$ being a large enough prime number. Then 
$$\chi_{\mathcal{A}}(q)=q^n-\# \cup_{H\in\mathcal{A}(q)}H.$$
\end{Theorem}


For any simple graph $G$ with vertices $[n]$ and edges $E(G)$, the associated \textit{graphical arrangement} $\mcA_G$ is defined to be: $x_i-x_j=0$ for $ij\in E(G)$. Every graphical arrangement is a sub-arrangement of the \textit{Braid arrangement}:
$$\mathcal{B}_n=\left\{x_i=x_j\mid1\leq i<j\leq n\right\}$$
whose graph is the complete graph with $n$ vertices. 

Given a set of parameters $A=\{(a_1,b_1),\ldots,(a_m,b_m)\}\subset\mbZ\times\mbZ$, we generalize this construction to the following type of arrangements:
$$\mathcal{A}_n=\mcA_n(A)\coloneqq\mathcal{B}_n\cup\left\{x_i=a_kx_j+b_k\mid1\leq i\neq j\leq n, 1\leq k\leq m\right\}.$$
We define the associated graph $G(\mcA_n,q)$ to have vertices $[q]$ and the edges $ij\in E(G(\mcA_n,q))$ iff $i\equiv a_k j+b_k(\text{mod }q)$ for some $k$.

The following proposition is an immediate result of Theorem \ref{finite field}, but is important in relating invariants of hyperplane arrangements to the corresponding graph.
\begin{Proposition}\label{24}
Let $\mcA_n=\mcA_n(A)$ and $G(\mcA_n,q)$ be defined as above for $q=p^r$ where $p$ is a large enough prime number. Then the number of $n$-element independent set of $G(\mcA_n,q)$ is $\chi_{\mathcal{A}_n}(q)/n!$.
\end{Proposition}
\begin{proof}
By construction, any ordered $n$-element independent set $(i_1,\ldots,i_n)\subset[q]^n$ in $G(\mcA_n,q)$ can be identified with a point in $\mbF_q^n$ which is outside the union of all hyperplanes in $\mcA_n(q)$. Regardless of the order, we know by Theorem \ref{finite field} that the number of $n$-element independent set of $G(\mcA_n,q)$ is $$\frac{q^n-\# \cup_{H\in\mathcal{A}_n(q)}H}{n!}=\frac{\chi_{\mathcal{A}_n}(q)}{n!}.$$
\end{proof}


We consider several special cases of the parameters $A$:
\begin{enumerate}
\item $A_0=\{(1,1)\}$. Then $\mcA_n(A_0)$ is exactly the Catalan arrangement:
$$\mathcal{C}_n=\{x_i-x_j=-1,0,1\mid 1\leq i<j\leq n\},$$
and it's well-known that 
$$\chi_{\mathcal{C}_n}(t)=t(t-n-1)_{n-1}$$
where $(k)_l=k(k-1)\cdots(k-l+1)$.
\item $A_1=\{(0,0),(a_1,0),\ldots,(a_m,0)\}$, then
$$\mcA_n(A_1)=\mathcal{B}_n\cup\{x_i=0\mid1\leq i\leq n\}\cup\{x_i=a_kx_j\mid 1\leq i\neq j\leq n,1\leq k\leq m\}$$
is a central arrangement. Since there are hyperplanes $x_i=0$, the vertex $q\in[q]$ is connected to all the other vertices in $G(\mcA_n(A_1),q)$. So when considering the independence sets, it suffices to look at a reduced graph $\tilde{G}(\mcA_n(A_1),q-1)$ with vertex set $[q-1]$ and edges $ij$ iff $i\equiv a_k j(\text{mod }q)$ for some $k$, which comes from $G(\mcA_n(A_1),q)$ by removing the vertex $q$ and its incident edges.
\item $A_2=\{(1,b_1),\ldots,(1,b_m)\}$, then
$$\mcA_n(A_2)=\mathcal{B}_n\cup\{x_i=x_j+b_k\mid 1\leq i\neq j\leq n,1\leq k\leq m\},$$
which is a deformation of the braid arrangement. In particular, $G(\mcA_n(A_2),q)$ has vertex set $[q]$ and edges $ij$ iff $i-j\equiv b_k$ for some $k$.
\end{enumerate}

Finally, we introduce the notion of exponential sequences of arrangements (ESA) following Section 5.3 \cite{stanley1}.
\begin{Definition}
A sequence of hyperplane arrangements $\mathfrak{A}= (\mcA_1, \mcA_2, . . . )$ is called an exponential sequence of arrangements if it satisfies the following:
\begin{enumerate}
\item $\mcA_n$ is in $\mathbb{R}^n$.
\item Every $H \in\mcA_n$ is parallel to some hyperplane in the braid arrangement $\mathcal{B}_n$. 
\item Let $S\subset[n]$ and 
$$\mcA_n^S=\{H\in\mcA_n\mid H\text{ is parallel to }x_i-x_j=0\text{ for some }i,j\in S\}. $$
Then $L(\mcA^S_n) \cong L(\mcA_{|S|})$.
\end{enumerate}
\end{Definition}
\begin{Theorem}[{\cite[Theorem 5.17]{stanley1}}]\label{esa}
$$\sum^{\infty}_{n=0}\chi_{\mathcal{A}_n}(t)\frac{x^n}{n!}=\left(\sum^{\infty}_{n=0}(-1)^nr(\mathcal{A}_n)\frac{x^n}{n!}\right)^{-t}=\exp\left(\sum_{n=1}^\infty \tilde{\chi}_{\mcA_n}(t)\frac{x^n}{n!}\right)$$
where $\tilde{\chi}_{\mcA_n}(t)=c_nt$ for some $c_n$.
\end{Theorem}
Note that under our construction, $(\mcA_n(A_0))_n$ and $(\mcA_n(A_2))_n$ are ESAs and we will use Theorem \ref{esa} to study their properties in Section 4.

\section{Generating Functions of the Characteristic Polynomials}

In this section, denote $\mathbb{N}$ as the set of natural numbers and $\mathbb{Q}$ as rational numbers. We consider the special parameter set $A_1$ defined in the previous section. In addition, we say $a_1,\ldots,a_m\in\mathbb{Q}_{>0}$ are \textit{multiplicatively independent} if $a_1^{n_1}\cdots a_m^{n_m}=1$ for some integers $n_1,\ldots,n_m$ implies $n_1=\cdots=n_m=0$.

\begin{Theorem}\label{thm31}
For fixed $m$, $n$ and multiplicatively independent parameters $a_1,\ldots,a_m\in\mathbb{N}_{>1}$, let $A_1$ and $\mcA_n=\mcA_n(A_1)$ be defined as before:
$$\mcA_n=\mcA_n(A_1)=\mathcal{B}_n\cup\{x_i=0\mid1\leq i\leq n\}\cup\{x_i=a_kx_j\mid 1\leq i\neq j\leq n,1\leq k\leq m\}.$$
Then for fixed $m$ and $n$, the characteristic polynomial $\chi_{\mathcal{A}_n}$ and in particular the number of regions $r(\mathcal{A}_n)$ are independent of the $a_i$'s.
\end{Theorem}
\begin{proof} 
Let $\mathcal{A}_n^\prime=\mathcal{A}_n-\{x_i=0 \mid 1\leq i\leq n\}$. Repeated use of the deletion-restriction method in \cite[Lemma 2.2]{stanley1} gives:\begin{equation}\label{1}
\chi_{\mcA_n^\prime}(t)=\chi_{\mathcal{A}_n}(t)+n\chi_{\mathcal{A}_{n-1}}(t).
\end{equation}

From this recursive formula, it suffices to show that $\chi_{\mcA_n^\prime}$ is independent of the $a_i$'s, as long as they are multiplicatively independent.

We now adopt a similar argument as in \cite[Theorem 5.17]{stanley1} to calculate the characteristic polynomial of $\mathcal{A}'_n$. By Theorem \ref{Whitney}, we have
$$\chi_{\mcA_n^\prime}(t)=\sum_{\mathcal{B}\subseteq\mathcal{A}_n^\prime}(-1)^{\#\mathcal{B}}t^{n-\rank(\mathcal{B})}.$$ 

For each sub-arrangement $\mcB\subseteq\mathcal{A}_n^\prime$, we associate a generalized weighted graph $G_\mcB$ with vertex set $[n]$ and we connect $i$ and $j$ by an (undirected) edge with weight 1 if $x_i=x_j\in\mcB$, and connect $i$ and $j$ by an arrow from $i$ to $j$ with  weight $a_l$ if $x_j=a_lx_i\in\mcB$ for some $i\neq j$ and $a_l\neq 1$.  Hence there are $2m+1$ possible types of edges or arrows between $i$ and $j$ ($i\neq j$). Let $\Pi_\mcB\in\Pi_n$ be the partition of $[n]$ corresponding to $\mcB$, where each block contains the vertices in a connected component of $G_\mcB$. By exponential formula \cite[\S 5.1]{stanley3}, we have
\begin{equation}\label{2}
\sum_{n=0}^\infty\chi_{\mcA^\prime_n}(t)\frac{x^n}{n!}=\exp\left(\sum_{n=1}^\infty f_n(t)\frac{x^n}{n!}\right)
\end{equation}
where $$f_n(t)=\sum_{\substack{\mathcal{B}\subseteq\mathcal{A}^\prime_n\\ \Pi_{\mathcal{B}}=[n]}}(-1)^{\mathcal{\#B}}t^{n-\rank(\mathcal{B})}.$$

The conditions $\mathcal{B}\subseteq\mathcal{A}^\prime_n$ and $\Pi_\mcB=[n]$ are equivalent to $G_\mcB $ being connected hence equivalent to $\rank\mcB\geq n-1$ (the solution of the linear system $\mathcal{B}$ is at most one dimensional). Therefore $f_n(t)$ is a polynomial with degree at most 1. $\chi_{\mcA_0^\prime}(1)=\chi_{\mcA_1^\prime}(1)=1$ and by Theorem \ref{Zaslavsky}, $\chi_{\mcA_n^\prime}(1)=0$ for $n>1$ since $\mcA_n^\prime$ is central. Plugging $t=1$ into Equation \eqref{2}, we have 
$$1+x=\exp\left(\sum_{n=1}^\infty f_n(1)\frac{x^n}{n!}\right)$$
Hence $f_n(1)=(-1)^{n-1}(n-1)!$ and
\begin{equation}\label{3}
f_n(t)=b_n(t-1)+(-1)^{n-1}(n-1)!
\end{equation}
where $$b_n=\sum_{\substack{\mcB\subseteq\mcA_n^\prime\\\rank(\mcB)=n-1}}(-1)^{\#\mcB}=\sum_{\substack{\mcB\subseteq\mcA_n^\prime\\ \rank(\mcB)=n-1}}(-1)^{\#E(G_\mcB)}.$$

For $\mcB\subset\mcA_n^\prime$, the condition $\rank(\mcB)=n-1$ is equivalent to saying the solution space of the linear system $\mathcal{B}$ is 1 dimensional. Since $G_{\mathcal{B}}$ is connected, a nontrivial solution to the linear system $\mathcal{B}$ exists if and only if for each cycle in $G_\mcB$, the product of the weights appearing on arrows in the clockwise direction must equal that in the counterclockwise direction in this cycle. Since $a_i$'s are multiplicatively independent, the arrows with weight $a_i$ in the clockwise direction has the same number  as those in the counterclockwise direction for each cycle in $G_{\mathcal{B}}.$ An example of such a $\mathcal{B}\subseteq\mathcal{A}'_n$ is given in Figure \ref{fig0}.
\begin{figure}
\centering
\begin{tikzpicture}[
            > = stealth, 
            shorten > = 1pt, 
            auto,
            node distance = 3cm, 
            semithick 
        ]

        \tikzstyle{every state}=[
            draw = black,
            thick,
            fill = white,
            minimum size = 4mm
        ]

        \node[state] (s) {1};
        \node[state] (v1) [above right of=s] {2};
        \node[state] (v2) [right of=s] {3};
        \node[state] (v3) [below of=s] {4};
        \node[state] (t) [right of=v2] {5};
        \node[state] (k) [below of=v2] {6};

        \path[->] (s) edge node {3} (v1);
        \path[->] (s) edge node {3} (v2);
        \path[->] (s) edge node {2} (v3);
        \path[-] (v2) edge node {1} (v1);
        \path[->] (v3) edge node {3} (k);
        \path[-] (v1) edge node {1} (t);
        \path[->] (v2) edge node {2} (k);
\end{tikzpicture}
\caption{This figure above is the graph $G_{\mathcal{B}}$ where $a_1=2$, $a_2=3$ and $\mathcal{B}=\{x_2=x_5,x_2=3x_1,x_2=x_3,x_3=3x_1,x_4=2x_1,x_6=3x_4,x_6=2x_3\}\subseteq \mathcal{A}'_n$. The solution of the linear system $\mathcal{B}$ is exactly 1 dimensional, i.e. $\rank(\mathcal{B})=n-1$. If we add a weight 1 edge between vertices 1 and 6, $\rank(\mathcal{B})=n$ and the cycle consisting of vertices 1, 3, 6 (or 1, 4, 6) violates the conditions described in the last two paragraphs of the proof.}
\label{fig0}
\end{figure}

Hence the above sum of $(-1)^{\#E(G_\mcB)}$ is actually taken over connected generalized graphs on $[n]$ with $m+1$ different weights (or $2m+1$ types of edges/arrows between each pair of vertices) with the described cycle condition. This value has nothing to do with specific values of $a_i$'s, as long as they are multiplicatively independent. Hence we arrive at the conclusion.

\end{proof}

Although $(\mcA_n(A_1))_n$ is not a ESA, we have the following result analogous to Theorem \ref{esa}.
\begin{Corollary}\label{cor32}
Let $A_1 $ and $\mathcal{A}_n=\mathcal{A}_n(A_1)$ be as before, we have
$$\sum^{\infty}_{n=0}\chi_{\mathcal{A}_n}(t)\frac{x^n}{n!}=\left(\sum^{\infty}_{n=0}(-1)^nr(\mathcal{A}_n)\frac{x^n}{n!}\right)^{-\frac{t-1}{2}}.$$
\end{Corollary}
\begin{proof}
Equation \eqref{1} shows that $$\sum^{\infty}_{n=0}\chi_{\mathcal{A}_n}(t)\frac{x^n}{n!}=\frac{\sum^{\infty}_{n=0}\chi_{\mathcal{A}_n^\prime}(t)\frac{x^n}{n!}}{1+x}.$$
By Equation \eqref{2} and \eqref{3}, we have:
$$\sum^{\infty}_{n=0}\chi_{\mathcal{A}_n^\prime}(t)\frac{x^n}{n!}=\exp\left(\sum_{n=1}^\infty (b_n(t-1)+(-1)^{n-1}(n-1)!)\frac{x^n}{n!}\right)=(1+x)\exp\left(\sum_{n=1}^\infty b_n(t-1)\frac{x^n}{n!}\right).$$
Hence by Theorem \ref{Zaslavsky}, we have
\begin{align*}
\sum^{\infty}_{n=0}\chi_{\mathcal{A}_n}(t)\frac{x^n}{n!} & =\exp\left(\sum_{n=1}^\infty b_n(t-1)\frac{x^n}{n!}\right)\\
 & =\left(\exp\Bigg(-\sum_{n=1}^\infty 2 b_n\frac{x^n}{n!}\Bigg)\right)^{-\frac{t-1}{2}}\\
 & = \left(\sum^{\infty}_{n=0}\chi_{\mathcal{A}_n}(-1)\frac{x^n}{n!}\right)^{-\frac{t-1}{2}}\\
 & =\left(\sum^{\infty}_{n=0}(-1)^nr(\mathcal{A}_n)\frac{x^n}{n!}\right)^{-\frac{t-1}{2}}.
\end{align*}
\end{proof}

\section{Counting Independent Sets in Graphs}
Before dealing with independent sets in disjoint unions of sets of type $\tilde{G}(\mcA_n(A_1),q-1)$ defined in Section 2, we first consider the Catalan arrangements, $\mathcal{C}_n=\mathcal{A}_n(A_0)$ where $A_0=\{(1,1)\}$. The associated graph $G(\mathcal{C}_n,q)$ is therefore the cycle $C_q$ with $q$ vertices. By Proposition \ref{24}, we know the number $n$-element independent sets of $C_q$ is $$\frac{\chi_{\mathcal{C}_n}(q)}{n!}=\frac{q(q-n-1)_{n-1}}{n!}.$$

Let $G=C_{q_1}\cup C_{q_2}$ be the disjoint union of two cycles and denote by $s_n$ the number of n-element independent sets of $G$. The independence polynomial $I_G(x)=\sum_{n=0}^\infty s_nx^n$ satisfies $I_G(x)=I_{C_{q_1}}(x)I_{C_{q_2}}(x)$. Then we know
$$ s_n=\sum^{n}_{j=0}\frac{\chi_{\mathcal{C}_{n-j}}(q_1)}{(n-j)!}\cdot\frac{\chi_{\mathcal{C}_{j}}(q_2)}{j!}.$$
Since the Catalan arrangements form an exponential sequence of arrangements, by Theorem \ref{esa} we have 
$$\sum^{\infty}_{n=0}\chi_{\mathcal{C}_n}(q_{i})\frac{x^n} {n!}=\exp\left(\sum^{\infty}_{n=1}\tilde{\chi}_{\mathcal{C}_n}(q_{i})\frac{x^n}{n!}\right),\;\;i=1,2$$
where $\tilde{\chi}_{\mathcal{C}_n}(t)=c_nt$. Then
$$I_G(x) =\exp\left(\sum^{\infty}_{n = 1}(\tilde{\chi}_{\mathcal{C}_n}(q_1)+\tilde{\chi}_{\mathcal{C}_n}(q_2))\frac{x^n}{n!}\right)
=\exp\left(\sum^{\infty}_{n = 1}c_n(q_1+q_2)\frac{x^n}{n!}\right)
=\sum^{\infty}_{n=0}\chi_{\mathcal{C}_n}(q_1+q_2)\frac{x^n}{n!}.$$
Hence $s_n=\frac{\chi_{\mathcal{C}_n}(q_1+q_2)}{n!}=\frac{q_1+q_2}{n!}(q_1+q_2-n-1)_{n-1},$ which means the number of $n$-independent sets of a disjoint union $G=C_{q_1}\cup C_{q_2}$ is equal to that of $C_{q_1+q_2}.$ By induction, we know the number of $n$-independent sets of $C_{q_1}\cup \ldots\cup  C_{q_s}$ only depends on $n$ and $\sum q_i$ when $q_i$'s are powers of a large enough prime number.

Things gets more complicated when $x_i=0,\;1\leq i\leq n$, are contained in the arrangement, in which case $\mathcal{A}_n$ no longer forms an exponential sequence. We now study the disjoint union of graphs $G(\mcA_n(A_1),q_1) \cup\cdots\cup G(\mcA_n(A_1),q_s)$, or just $\tilde{G}(\mcA_n(A_1),q_1-1)\cup\cdots\cup \tilde{G}(\mcA_n(A_1),q_s-1)$. Recall from Section 2 that $\tilde{G}(\mcA_n(A_1),q-1)$ has vertex set $[q-1]$ and edges $ij$ iff $i\equiv a_k j$ for some $1\leq k\leq m$.

\begin{Theorem}
For multiplicatively independent parameters $a_1,\ldots,a_m\in\mathbb{N}_{>1}$, let $A_1$ and $\mcA_n=\mcA_n(A_1)$ be defined as in Section 2:
$$\mcA_n=\mcA_n(A_1)=\mathcal{B}_n\cup\{x_i=0\mid1\leq i\leq n\}\cup\{x_i=a_kx_j\mid 1\leq i\neq j\leq n,1\leq k\leq m\}.$$ 
Let G be the disjoint union $\tilde{G}(\mcA_n,q_1-1)\cup\cdots\cup\tilde{G}(\mcA_n,q_s-1)$ where $q_1,\ldots,q_s$ are powers of sufficiently large prime numbers. Then the number of $n$-element independent sets of $G$ depends only on $n$, $m,$ and the total number of vertices $\sum_{i=1}^s(q_i-1).$
\end{Theorem}
\begin{proof} 
By Proposition \ref{24}, the number $s_n$ of $n$-element independent sets of $G$ is calculated as: 
$$s_n=\sum_{\substack{j_1,j_2,...,j_s\geq 0\\j_1+j_2+\cdots+j_s=n}}\frac{\chi_{\mathcal{A}_{j_1}}(q_1)}{j_1!}\cdots\frac{\chi_{\mathcal{A}_{j_s}}(q_s)}{j_s!}.$$ 
The independence polynomial is obtained as
$$I_G(x)=\sum^{\infty}_{n=0}s_n x^n=\sum^{\infty}_{n=0}\chi_{\mathcal{A}_n}(q_1)\frac{x^n} {n!}\cdots \sum^{\infty}_{n=0}\chi_{\mathcal{A}_n}(q_s)\frac{x^n}{n!}.$$
Applying Corollary \ref{cor32}, we have
\begin{align*}
I_G(x) & =\exp\left(\sum^{\infty}_{n=1}b_n(q_1+\cdots+q_s-s)\frac{x^n}{n!}\right)\\
& = \sum^{\infty}_{n=0}\chi_{\mathcal{A}_n}(q_1+\cdots+q_s-s+1)\frac{x^n}{n!},
\end{align*}
which implies that $s_n=\chi_{\mathcal{A}_n}(q_1+\cdots+q_s-s+1)/n!.$ 

The conclusion follows after applying Theorem \ref{thm31}.
\end{proof}

\begin{Example}
Let $a_1=2,a_2=3$ and $G_k$ be the graph with vertex set $[k]$ and edges $ij$ if $i \equiv 2j\;(\textnormal{mod}\; k+1), $ or $i \equiv 3j\;(\textnormal{mod}\; k+1)$.

Then the number of independent sets of the disjoint union of $G_{18}$ and $G_{22}$ is equal to that of the graph $G_{40},$ as shown in Figure \ref{fig1}. Note that the number of independent sets of the disjoint union is independent of $a_i$'s, since they are multiplicatively independent.
\end{Example}

\begin{figure}
\centering
\begin{tikzpicture}[x=0.50cm,y=0.50cm]
\clip(-1.4,-2.5) rectangle (34,11.8);
\draw [thick] (8,5)--(7.759,6.368)--(2,8.464);
\draw [thick] (7.759,6.368)--(7.06,7.571)-- (0.936,7.571);
\draw [thick] (7.064,7.571)--(6,8.464)-- (0.241,6.368);
\draw [thick] (6,8.464)--(4.695,8.939)-- (0,5);
\draw [thick] (4.695,8.939)--(3.305,8.939)-- (0.241,3.632);
\draw [thick] (3.305,8.939)--(2,8.464)-- (0.936,2.429);
\draw [thick] (2,8.464)--(0.936,7.571)-- (2,1.536);
\draw [thick] (0.936,7.571)--(0.241,6.368)-- (3.305,1.061);
\draw [thick] (0.241,6.368)--(0,5)-- (4.695,1.061);
\draw [thick] (0,5)--(0.241,3.632)-- (6,1.536);
\draw [thick] (0.241,3.632)--(0.936,2.429)-- (7.064,2.429);
\draw [thick] (0.936,2.429)--(2,1.536)-- (7.759,3.632);
\draw [thick] (2,1.536)--(3.305,1.061)-- (8,5);
\draw [thick] (3.305,1.061)--(4.695,1.061)-- (7.759,6.368);
\draw [thick] (4.695,1.061)--(6,1.536)-- (7.064,7.571);
\draw [thick] (6,1.536)--(7.064,2.429)-- (6,8.464);
\draw [thick] (7.064,2.429)--(7.759,3.632)-- (4.695,8.939);
\draw [thick] (7.759,3.632)--(8,5)-- (3.305,8.939);
\draw [fill=black] (7.759,6.368) circle (1.7pt);
\draw [fill=black] (0,5) circle (1.7pt);
\draw [fill=black] (2,8.464) circle (1.7pt);
\draw [fill=black] (0.936,7.571) circle (1.7pt);
\draw [fill=black] (0.241,6.368) circle (1.7pt);
\draw [fill=black] (0.241,3.632) circle (1.7pt);
\draw [fill=black] (3.305,8.939) circle (1.7pt);
\draw [fill=black] (8,5) circle (1.7pt);
\draw [fill=black] (7.759,3.632) circle (1.7pt);
\draw [fill=black] (4.695,8.939) circle (1.7pt);
\draw [fill=black] (7.064,2.429) circle (1.7pt);
\draw [fill=black] (6,1.536) circle (1.7pt);
\draw [fill=black] (4.695,1.061) circle (1.7pt);
\draw [fill=black] (3.305,1.061) circle (1.7pt);
\draw [fill=black] (6,8.464) circle (1.7pt);
\draw [fill=black] (7.064,7.571) circle (1.7pt);
\draw [fill=black] (2,1.536) circle (1.7pt);
\draw [fill=black] (0.936,2.429) circle (1.7pt);

\draw [thick] (16,2.5)--(15.603,3.852)--(11.863,4.389);
\draw [thick] (15.603,3.852)--(14.539,4.774)--(11.101,3.204);
\draw [thick] (14.539,4.774)--(13.144,4.975)--(11.101,1.796);
\draw [thick] (13.144,4.975)--(11.863,4.389)--(11.863,0.611);
\draw [thick] (11.863,4.389)--(11.101,3.204)--(13.144,0.025);
\draw [thick] (11.101,3.204)--(11.101,1.796)--(14.539,0.226);
\draw [thick] (11.101,1.796)--(11.863,0.611)--(15.603,1.148);
\draw [thick] (11.863,0.611)--(13.144,0.025)--(16,2.5);
\draw [thick] (13.144,0.025)--(14.539,0.226)--(15.603,3.852);
\draw [thick] (14.539,0.226)--(15.603,1.148)--(14.539,4.774);
\draw [thick] (15.603,1.148)--(16,2.5)--(13.144,4.975);
\draw [fill=black] (15.603,3.852) circle (1.7pt);
\draw [fill=black] (11.863,4.389) circle (1.7pt);
\draw [fill=black] (11.101,3.204) circle (1.7pt);
\draw [fill=black] (13.144,4.975) circle (1.7pt);
\draw [fill=black] (16,2.5) circle (1.7pt);
\draw [fill=black] (15.603,1.148) circle (1.7pt);
\draw [fill=black] (14.539,4.774) circle (1.7pt);
\draw [fill=black] (14.539,0.226) circle (1.7pt);
\draw [fill=black] (13.144,0.025) circle (1.7pt);
\draw [fill=black] (11.863,0.611) circle (1.7pt);
\draw [fill=black] (11.101,1.796) circle (1.7pt);

\draw [thick] (16,8.5)--(15.603,9.852)--(11.863,10.389);
\draw [thick] (15.603,9.852)--(14.539,10.774)--(11.101,9.204);
\draw [thick] (14.539,10.774)--(13.144,10.975)--(11.101,7.796);
\draw [thick] (13.144,10.975)--(11.863,10.389)--(11.863,6.611);
\draw [thick] (11.863,10.389)--(11.101,9.204)--(13.144,6.025);
\draw [thick] (11.101,9.204)--(11.101,7.796)--(14.539,6.226);
\draw [thick] (11.101,7.796)--(11.863,6.611)--(15.603,7.148);
\draw [thick] (11.863,6.611)--(13.144,6.025)--(16,8.5);
\draw [thick] (13.144,6.025)--(14.539,6.226)--(15.603,9.852);
\draw [thick] (14.539,6.226)--(15.603,7.148)--(14.539,10.774);
\draw [thick] (15.603,7.148)--(16,8.5)--(13.144,10.975);
\draw [fill=black] (15.603,9.852) circle (1.7pt);
\draw [fill=black] (11.863,10.389) circle (1.7pt);
\draw [fill=black] (11.101,9.204) circle (1.7pt);
\draw [fill=black] (13.144,10.975) circle (1.7pt);
\draw [fill=black] (16,8.5) circle (1.7pt);
\draw [fill=black] (15.603,7.148) circle (1.7pt);
\draw [fill=black] (14.539,10.774) circle (1.7pt);
\draw [fill=black] (14.539,6.226) circle (1.7pt);
\draw [fill=black] (13.144,6.025) circle (1.7pt);
\draw [fill=black] (11.863,6.611) circle (1.7pt);
\draw [fill=black] (11.101,7.796) circle (1.7pt);

\draw [thick] (27.536,9.036)--(26.270,9.955)--(23.227,7.878);
\draw [thick] (26.270,9.955)--(24.782,10.438)--(22.531,7.523);
\draw [thick] (24.782,10.438)--(23.218,10.438)--(21.977,6.969);
\draw [thick] (23.218,10.438)--(21.730,9.955)--(21.622,6.273);
\draw [thick] (21.730,9.955)--(20.464,9.036)--(21.5,5.5);
\draw [thick] (20.464,9.036)--(19.545,7.770)--(21.622,4.727);
\draw [thick] (19.545,7.770)--(19.062,6.282)--(21.977,4.031);
\draw [thick] (19.062,6.282)--(19.062,4.718)--(22.531,3.477);
\draw [thick] (19.062,4.718)--(19.545,3.230)--(23.227,3.122);
\draw [thick] (19.545,3.230)--(20.464,1.964)--(24,3);
\draw [thick] (20.464,1.964)--(21.730,1.045)--(24.773,3.122);
\draw [thick] (21.730,1.045)--(23.218,0.562)--(25.469,3.477);
\draw [thick] (23.218,0.562)--(24.782,0.56)--(26.023,4.031);2
\draw [thick] (24.782,0.562)--(26.270,1.045)--(26.378,4.727);
\draw [thick] (26.270,1.045)--(27.536,1.964)--(26.5,5.5);
\draw [thick] (27.536,1.964)--(28.455,3.230)--(26.378,6.273);
\draw [thick] (28.455,3.230)--(28.938,4.718)--(26.023,6.969);
\draw [thick] (28.938,4.718)--(28.938,6.282)--(25.469,7.523);
\draw [thick] (28.938,6.282)--(28.455,7.770)--(24.773,7.878);
\draw [thick] (28.455,7.770)--(27.536,9.036)--(24,8);

\draw [thick] (26.5,5.5)--(26.378,6.273)--(26.270,9.955);
\draw [thick] (26.378,6.273)--(26.023,6.969)--(24.782,10.438);
\draw [thick] (26.023,6.969)--(25.469,7.523)--(23.218,10.438);
\draw [thick] (25.469,7.523)--(24.773,7.878)--(21.730,9.955);
\draw [thick] (24.773,7.878)--(24,8)--(20.464,9.036);
\draw [thick] (24,8)--(23.227,7.878)--(19.545,7.770);
\draw [thick] (23.227,7.878)--(22.531,7.523)--(19.062,6.282);
\draw [thick] (22.531,7.523)--(21.977,6.969)--(19.062,4.718);
\draw [thick] (21.977,6.969)--(21.622,6.273)--(19.545,3.230);
\draw [thick] (21.622,6.273)--(21.5,5.5)--(20.464,1.964);
\draw [thick] (21.5,5.5)--(21.622,4.727)--(21.730,1.045);
\draw [thick] (21.622,4.727)--(21.977,4.031)--(23.218,0.562);
\draw [thick] (21.977,4.031)--(22.531,3.477)--(24.782,0.56);
\draw [thick] (22.531,3.477)--(23.227,3.122)--(26.270,1.045);
\draw [thick] (23.227,3.122)--(24,3)--(27.536,1.964);
\draw [thick] (24,3)--(24.773,3.122)--(28.455,3.230);
\draw [thick] (24.773,3.122)--(25.469,3.477)--(28.938,4.718);
\draw [thick] (25.469,3.477)--(26.023,4.031)--(28.938,6.282);
\draw [thick] (26.023,4.031)--(26.378,4.727)--(28.455,7.770);
\draw [thick] (26.378,4.727)--(26.5,5.5)--(27.536,9.036);

\draw [fill=black] (27.536,9.036) circle (1.7pt);
\draw [fill=black] (26.270,9.955) circle (1.7pt);
\draw [fill=black] (24.782,10.438) circle (1.7pt);
\draw [fill=black] (23.218,10.438) circle (1.7pt);
\draw [fill=black] (21.730,9.955) circle (1.7pt);
\draw [fill=black] (20.464,9.036) circle (1.7pt);
\draw [fill=black] (19.545,7.770) circle (1.7pt);
\draw [fill=black] (19.062,6.282) circle (1.7pt);
\draw [fill=black] (19.062,4.718) circle (1.7pt);
\draw [fill=black] (19.545,3.230) circle (1.7pt);
\draw [fill=black] (20.464,1.964) circle (1.7pt);
\draw [fill=black] (21.730,1.045) circle (1.7pt);
\draw [fill=black] (23.218,0.562) circle (1.7pt);
\draw [fill=black] (24.782,0.562) circle (1.7pt);
\draw [fill=black] (26.270,1.045) circle (1.7pt);
\draw [fill=black] (27.536,1.964) circle (1.7pt);
\draw [fill=black] (28.455,3.230) circle (1.7pt);
\draw [fill=black] (28.938,4.718) circle (1.7pt);
\draw [fill=black] (28.938,6.282) circle (1.7pt);
\draw [fill=black] (28.455,7.770) circle (1.7pt);

\draw [fill=black] (26.5,5.5) circle (1.7pt);
\draw [fill=black] (26.378,6.273) circle (1.7pt);
\draw [fill=black] (26.023,6.969) circle (1.7pt);
\draw [fill=black] (25.469,7.523) circle (1.7pt);
\draw [fill=black] (24.773,7.878) circle (1.7pt);
\draw [fill=black] (24,8) circle (1.7pt);
\draw [fill=black] (23.227,7.878) circle (1.7pt);
\draw [fill=black] (22.531,7.523) circle (1.7pt);
\draw [fill=black] (21.977,6.969) circle (1.7pt);
\draw [fill=black] (21.622,6.273) circle (1.7pt);
\draw [fill=black] (21.5,5.5) circle (1.7pt);
\draw [fill=black] (21.622,4.727) circle (1.7pt);
\draw [fill=black] (21.977,4.031) circle (1.7pt);
\draw [fill=black] (22.531,3.477) circle (1.7pt);
\draw [fill=black] (23.227,3.122) circle (1.7pt);
\draw [fill=black] (24,3) circle (1.7pt);
\draw [fill=black] (24.773,3.122) circle (1.7pt);
\draw [fill=black] (25.469,3.477) circle (1.7pt);
\draw [fill=black] (26.023,4.031) circle (1.7pt);
\draw [fill=black] (26.378,4.727) circle (1.7pt);

\draw[color=black] (4,-1.5) node {$G_{18}$};
\draw[color=black] (13.5,-1.5) node {$G_{22}$};
\draw[color=black] (24,-1.5) node {$ G_{40}$};
\draw[color=black] (8.5 ,5) node {\hspace{10mm}\Huge{ $\cup$}};

\draw[color=black] (3.4,.55) node {1};
\draw[color=black] (2,1) node {2};
\draw[color=black] (.7,2.0) node {4};
\draw[color=black] (-0.1,3.4) node {8};
\draw[color=black] (-0.6,4.9) node {16};
\draw[color=black] (-0.3,6.43) node {13};
\draw[color=black] (.6,7.7) node {7};
\draw[color=black] (1.8,8.9) node {14};
\draw[color=black] (3.3,9.5) node {9};
\draw[color=black] (4.8,9.4) node {18};
\draw[color=black] (6.3,8.9) node {17};
\draw[color=black] (7.6,7.7) node {15};
\draw[color=black] (8.4,6.43) node {11};
\draw[color=black] (8.5,4.9) node {3};
\draw[color=black] (8.2,3.4) node {6};
\draw[color=black] (7.4,1.9) node {12};
\draw[color=black] (6.1,1) node {5};
\draw[color=black] (4.7,.55) node {10};

\draw[color=black] (11.45,6.4) node {2};
\draw[color=black] (10.7,7.7) node {4};
\draw[color=black] (10.7,9.2) node {8};
\draw[color=black] (11.3,10.5) node {16};
\draw[color=black] (13,11.4) node {9};
\draw[color=black] (14.7,11.2) node {18};
\draw[color=black] (16.1,10) node {13};
\draw[color=black] (16.4,8.5) node {3};
\draw[color=black] (16.1,7) node {6};
\draw[color=black] (14.9,5.85) node {12};
\draw[color=black] (13.5,5.7) node {1};

\draw[color=black] (11.45,6.4-6.1) node {21};
\draw[color=black] (10.6,7.7-6.1) node {19};
\draw[color=black] (10.6,9.2-6.1) node {15};
\draw[color=black] (11.4,10.5-6.1) node {7};
\draw[color=black] (12.7,11.4-6.1) node {14};
\draw[color=black] (14.9,11.1-6.1) node {5};
\draw[color=black] (16.1,10-6.1) node {10};
\draw[color=black] (16.6,8.6-6.1) node {20};
\draw[color=black] (16.1,7-6.1) node {17};
\draw[color=black] (14.7,5.85-6.1) node {11};
\draw[color=black] (12.9,-.4) node {22};

\draw[color=black] (20.1,9.2) node {1};
\draw[color=black] (21.5,10.3) node {2};
\draw[color=black] (23.1,10.9) node {4};
\draw[color=black] (25,10.8) node {8};
\draw[color=black] (26.7,10.2) node {16};
\draw[color=black] (28.1,9.2) node {32};
\draw[color=black] (29,7.9) node {23};
\draw[color=black] (29.4,6.25) node {5};
\draw[color=black] (29.5,4.6) node {10};
\draw[color=black] (29,2.9) node {20};
\draw[color=black] (27.8,1.6) node {40};
\draw[color=black] (26.6,.6) node {39};
\draw[color=black] (25,.1) node {37};
\draw[color=black] (23.5,.1) node {33};
\draw[color=black] (21.8,.5) node {25};
\draw[color=black] (20.2,1.6) node {9};
\draw[color=black] (19.1,2.9) node {18};
\draw[color=black] (18.5,4.6) node {36};
\draw[color=black] (18.5,6.2) node {31};
\draw[color=black] (19.1,7.9) node {21};

\draw[color=black] (24,7.7) node {\tiny{14}};
\draw[color=black] (24.7,7.52) node {\tiny{28}};
\draw[color=black] (25.2,7.2) node {\tiny{15}};
\draw[color=black] (25.6,6.8) node {\tiny{30}};
\draw[color=black] (25.93,6.2) node {\tiny{19}};
\draw[color=black] (26.1,5.5) node {\tiny{38}};
\draw[color=black] (25.9,4.8) node {\tiny{35}};
\draw[color=black] (25.65,4.2) node {\tiny{29}};
\draw[color=black] (25.3,3.8) node {\tiny{17}};
\draw[color=black] (24.8,3.5) node {\tiny{34}};
\draw[color=black] (24,3.35) node {\tiny{27}};
\draw[color=black] (23.3,3.5) node {\tiny{13}};
\draw[color=black] (22.8,3.8) node {\tiny{26}};
\draw[color=black] (22.35,4.2) node {\tiny{11}};
\draw[color=black] (22.1,4.75) node {\tiny{22}};
\draw[color=black] (21.85,5.5) node {\tiny{3}};
\draw[color=black] (21.92,6.2) node {\tiny{6}};
\draw[color=black] (22.4,6.8) node {\tiny{12}};
\draw[color=black] (22.9,7.3) node {\tiny{24}};
\draw[color=black] (23.3,7.55) node {\tiny{7}};
\end{tikzpicture}
\caption{}
\label{fig1}
\end{figure}

Now we turn to the third category $\{\mcA_n(A_2)\}_n$ mentioned in Section 2 where $A_2=\{(1,b_1),\ldots,(1,b_m)\}$ and
$$\mcA_n(A_2)=\mathcal{B}_n\cup\{x_i=x_j+b_k\mid 1\leq i\neq j\leq n,1\leq k\leq m\}.$$

\begin{Lemma}\label{lem42}
In the setting above, $G(\mcA_n(A_2),q)$ is the graph with vertex set $[q]$ and edges $ij$ iff $i-j\equiv b_k$ for some $k$. Then for any sufficiently large enough $q$ (not necessarily a prime power), the number of $n\con$element independent sets of $G(\mcA_n(A_2),q)$ is $\frac{\chi_{\mathcal{A}_n(A_2)}(q)}{n!}$.
\end{Lemma}
\begin{proof}
Proposition \ref{24} already shows that this result is true for $q$ being a large enough prime power, so it suffices to show that the number of $n$-element independent sets can be expressed as a polynomial in $q$, so must be $\frac{\chi_{\mathcal{A}_n(A_2)}(q)}{n!}.$

Similar to the proof of Proposition \ref{24}, we are still counting the number of tuples $(i_1,\ldots,i_n)\in[q]^n$ that's outside the union of hyperplanes in $\mathcal{A}_n(A_2)$. By inclusion-exclusion principle, we know that this number is the signed sum of the number of elements in $H_1\cap\cdots\cap H_k$,  $H_1,\ldots,H_k\in\mathcal{A}_n(A_2)$. This is counting the solution of a linear system defined in $\mathbb{Z}/q\mathbb{Z}$, where $q$ is not necessarily a prime power. But the normals of our hyperplanes are vectors with only one 1 and -1, and all other entries are zero, so the Gaussian elimination procedure to obtain the row reduced echelon form(RREF) does not involve multiplication and division. Thus as long as $q$ is large enough, $|H_1\cap\cdots\cap H_k|$ is equal to 0 or $q^{\#\text{free variables in the RREF}}$, which is a polynomial. So we arrived at the conclusion.
\end{proof}

\begin{Theorem}
In the same setting as above, let $F$ be the disjoint union $G(\mcA_n(A_2),q_1)\cup\cdots\cup G(\mcA_n(A_2),q_s)$ where $q_i$'s are large enough (but not necessarily prime powers). Then the number of $n$-element independent sets in $F$ depends only on $b_i$'s, $n,$ and $\sum_{i=1}^s{q_i}.$
\end{Theorem}
\begin{proof}
For simplicity denote $\mcA_n=\mcA_n(A_2)$. Since $\mcA_n$ forms an exponential sequence, by Theorem \ref{esa} we have 
\begin{align*}
\sum^{\infty}_{n=0}\chi_{\mathcal{A}_n}(k_i)\frac{x^n}      {n!}&=\exp\left(\sum_{n\geq 1}\tilde{\chi}_{\mathcal{A}_n}(k_i)\frac{x^n}{n!}\right), \;\; i = 1,2,...,s
\end{align*}
where $\tilde{\chi}_{\mathcal{A}_n}(k_i)=u_n k_i, \;\; i = 1,2,...,s$. Here, $u_n$ only depends on $n$ and the $b_i$'s. 

Using Lemma \ref{lem42}, the independence polynomial of $F$ can be calculated as: 
\begin{align*} 
I_F(x)=\sum^{\infty}_{n=0}s_nx^n
&=\sum^{\infty}_{n=0}\sum_{\substack{j_1,...,j_s\geq 0\\j_1+\cdots+j_s=n}}\chi_{\mathcal{A}_{j_1}}(k_1)\cdots\chi_{\mathcal{A}_{j_s}}(k_s)\frac{x^n}{j_1!j_2!\cdots j_s!} \\
&=\sum^{\infty}_{n=0}\chi_{\mathcal{A}_n}(k_1)\frac{x^n}{n!}\cdots \sum^{\infty}_{n=0}\chi_{\mathcal{A}_n}(k_s)\frac{x^n}{n!} \\
&=\exp\left(\sum_{n\geq 1}(\tilde{\chi}_{\mathcal{A}_n}(k_1)+\tilde{\chi}_{\mathcal{A}_n}(k_2)+\cdots +\tilde{\chi}_{\mathcal{A}_n}(k_s))\frac{x^n}{n!}\right)\\
&=\sum^{\infty}_{n=0}\chi_{\mathcal{A}_n}(\sum_{i} k_i)\frac{x^n}{n!} .
\end{align*}
Hence the conclusion follows.
\end{proof}

\begin{Example}
Let $b_1=1,b_2=3$. Let $G_k$ be the graph with vertex set $[k]$ and edges $ij$ if $i -j\equiv 1\;(\textnormal{mod}\; k), $ or $i -j\equiv 3\;(\textnormal{mod}\; k),$ as shown in Figure \ref{fig2}. The number of $n$-element independent sets of the disjoint union $G = G_{k_1}\cup G_{k_2}$ is equal to that of the graph $G(k_1+k_2),$ which can be calculated by $\chi_{\mathcal{A}_n}(k_1+k_2)/n!.$ 
\end{Example}

\begin{figure}[h]
\centering
\begin{tikzpicture}[x=0.50cm,y=0.45cm]
\clip(-5.2,-1.8) rectangle (34,16.3);
\draw [dotted] (2,12) arc (30:60:2.3095);
\draw [dotted] (-2,12) arc (150:120:2.3095);
\draw [dotted] (2,4) arc (-30:-60:2.3095);
\draw [dotted] (-2,4) arc (-150:-120:2.3095);
\draw [thick] (-3,5)--(-3,8)--(-3,11);
\draw [thick] (3,5)--(3,8)--(3,11);
\draw [thick, dashed] (-3,11)--(-3,12);
\draw [thick, dashed] (-3,5)--(-3,4);
\draw [thick, dashed] (3,11)--(3,12);
\draw [thick, dashed] (3,5)--(3,4);
\draw [thick, dashed] (3,12) arc (15:165:3.1058);
\draw [thick, dashed] (3,4) arc (-15:-165:3.1058);
\draw [thick, dashed] (-3,4) arc (-90:0:1.5);
\draw [thick] (-1.5,5.5) arc (0:90:1.5);
\draw [thick] (-3,5) arc (-90:90:1.5);
\draw [thick] (-3,6) arc (-90:90:1.5);
\draw [thick] (-3,7) arc (-90:90:1.5);
\draw [thick] (-3,8) arc (-90:90:1.5);
\draw [thick] (-3,9) arc (-90:0:1.5);
\draw [thick, dashed] (-3,9) arc (-90:90:1.5);
\draw [thick, dashed] (3,12) arc (90:180:1.5);
\draw [thick] (1.5,10.5) arc (180:270:1.5);
\draw [thick] (3,11) arc (90:270:1.5);
\draw [thick] (3,10) arc (90:270:1.5);
\draw [thick] (3,9) arc (90:270:1.5);
\draw [thick] (3,8) arc (90:270:1.5);
\draw [thick] (3,7) arc (90:180:1.5);
\draw [thick, dashed] (3,7) arc (90:270:1.5);
\draw [fill=black] (-3,5) circle (1.5pt);
\draw [fill=black] (-3,6) circle (1.5pt);
\draw [fill=black] (-3,7) circle (1.5pt);
\draw [fill=black] (-3,8) circle (1.5pt);
\draw [fill=black] (-3,9) circle (1.5pt);
\draw [fill=black] (-3,10) circle (1.5pt);
\draw [fill=black] (-3,11) circle (1.5pt);
\draw [fill=black] (3,5) circle (1.5pt);
\draw [fill=black] (3,6) circle (1.5pt);
\draw [fill=black] (3,7) circle (1.5pt);
\draw [fill=black] (3,8) circle (1.5pt);
\draw [fill=black] (3,9) circle (1.5pt);
\draw [fill=black] (3,10) circle (1.5pt);
\draw [fill=black] (3,11) circle (1.5pt);

\draw [dotted] (13,11) arc (30:60:2.3095);
\draw [dotted] (9,11) arc (150:120:2.3095);
\draw [dotted] (13,5) arc (-30:-60:2.3095);
\draw [dotted] (9,5) arc (-150:-120:2.3095);
\draw [thick] (8,6)--(8,8)--(8,10);
\draw [thick] (14,6)--(14,8)--(14,10);
\draw [thick, dashed] (8,10)--(8,11);
\draw [thick, dashed] (8,6)--(8,5);
\draw [thick, dashed] (14,10)--(14,11);
\draw [thick, dashed] (14,6)--(14,5);
\draw [thick, dashed] (14,11) arc (15:165:3.1058);
\draw [thick, dashed] (14,5) arc (-15:-165:3.1058);
\draw [thick, dashed] (8,5) arc (-90:0:1.5);
\draw [thick] (9.5,6.5) arc (0:90:1.5);
\draw [thick] (8,6) arc (-90:90:1.5);
\draw [thick] (8,7) arc (-90:90:1.5);
\draw [thick] (8,8) arc (-90:0:1.5);
\draw [thick, dashed] (8,8) arc (-90:90:1.5);
\draw [thick, dashed] (14,11) arc (90:180:1.5);
\draw [thick] (12.5,9.5) arc (180:270:1.5);
\draw [thick] (14,10) arc (90:270:1.5);
\draw [thick] (14,9) arc (90:270:1.5);
\draw [thick] (14,8) arc (90:180:1.5);
\draw [thick, dashed] (14,8) arc (90:270:1.5);
\draw [fill=black] (8,6) circle (1.5pt);
\draw [fill=black] (8,7) circle (1.5pt);
\draw [fill=black] (8,8) circle (1.5pt);
\draw [fill=black] (8,9) circle (1.5pt);
\draw [fill=black] (8,10) circle (1.5pt);
\draw [fill=black] (14,10) circle (1.5pt);
\draw [fill=black] (14,9) circle (1.5pt);
\draw [fill=black] (14,8) circle (1.5pt);
\draw [fill=black] (14,7) circle (1.5pt);
\draw [fill=black] (14,6) circle (1.5pt);

\draw [dotted] (24,14) arc (30:60:2.3095);
\draw [dotted] (20,14) arc (150:120:2.3095);
\draw [dotted] (24,2) arc (-30:-60:2.3095);
\draw [dotted] (20,2) arc (-150:-120:2.3095);
\draw [thick] (19,3)--(19,8)--(19,13);
\draw [thick] (25,3)--(25,8)--(25,13);
\draw [thick, dashed] (19,13)--(19,14);
\draw [thick, dashed] (19,3)--(19,2);
\draw [thick, dashed] (25,13)--(25,14);
\draw [thick, dashed] (25,3)--(25,2);
\draw [thick, dashed] (25,14) arc (15:165:3.1058);
\draw [thick, dashed] (25,2) arc (-15:-165:3.1058);
\draw [thick, dashed] (19,2) arc (-90:0:1.5);
\draw [thick] (20.5,3.5) arc (0:90:1.5);
\draw [thick] (19,3) arc (-90:90:1.5);
\draw [thick] (19,4) arc (-90:90:1.5);
\draw [thick] (19,5) arc (-90:90:1.5);
\draw [thick] (19,6) arc (-90:90:1.5);
\draw [thick] (19,7) arc (-90:90:1.5);
\draw [thick] (19,8) arc (-90:90:1.5);
\draw [thick] (19,9) arc (-90:90:1.5);
\draw [thick] (19,10) arc (-90:90:1.5);
\draw [thick, dashed] (19,11) arc (-90:90:1.5);
\draw [thick] (19,11) arc (-90:0:1.5);
\draw [thick, dashed] (25,14) arc (90:180:1.5);
\draw [thick] (23.5,12.5) arc (180:270:1.5);
\draw [thick] (25,13) arc (90:270:1.5);
\draw [thick] (25,12) arc (90:270:1.5);
\draw [thick] (25,11) arc (90:270:1.5);
\draw [thick] (25,10) arc (90:270:1.5);
\draw [thick] (25,9) arc (90:270:1.5);
\draw [thick] (25,8) arc (90:270:1.5);
\draw [thick] (25,7) arc (90:270:1.5);
\draw [thick] (25,6) arc (90:270:1.5);
\draw [thick] (25,5) arc (90:180:1.5);
\draw [thick, dashed] (25,5) arc (90:270:1.5);
\draw [fill=black] (19,3) circle (1.5pt);
\draw [fill=black] (19,4) circle (1.5pt);
\draw [fill=black] (19,5) circle (1.5pt);
\draw [fill=black] (19,6) circle (1.5pt);
\draw [fill=black] (19,7) circle (1.5pt);
\draw [fill=black] (19,8) circle (1.5pt);
\draw [fill=black] (19,9) circle (1.5pt);
\draw [fill=black] (19,10) circle (1.5pt);
\draw [fill=black] (19,11) circle (1.5pt);
\draw [fill=black] (19,12) circle (1.5pt);
\draw [fill=black] (19,13) circle (1.5pt);
\draw [fill=black] (25,3) circle (1.5pt);
\draw [fill=black] (25,4) circle (1.5pt);
\draw [fill=black] (25,5) circle (1.5pt);
\draw [fill=black] (25,6) circle (1.5pt);
\draw [fill=black] (25,7) circle (1.5pt);
\draw [fill=black] (25,8) circle (1.5pt);
\draw [fill=black] (25,9) circle (1.5pt);
\draw [fill=black] (25,10) circle (1.5pt);
\draw [fill=black] (25,11) circle (1.5pt);
\draw [fill=black] (25,12) circle (1.5pt);
\draw [fill=black] (25,13) circle (1.5pt);

\draw[color=black] (0,-1.4) node {$G_{k_1}$};
\draw[color=black] (11,-1.4) node {$G_{k_2}$};
\draw[color=black] (22,-1.4) node {$G_{k_1+k_2}$};
\draw[color=black] (5.5,8) node {\Huge{$\cup$}};

\end{tikzpicture}
\caption{}
\label{fig2}
\end{figure}

We now consider the graph with additional edge attachment to the previous ones.
\begin{Corollary}

In the same setting as above, let $\bar{F}$ be the disjoint union $\bar{G}(\mcA_n(A_2),q_1)\cup\cdots\cup\bar{G}(\mcA_n(A_2),q_s)$ where $\bar{G}(\mcA_n(A_2),q)$ comes from $G(\mcA_n(A_2),q)$ by adding a copy of vertices $\{1',\ldots,q'\}$ and connect $j$ with $j'$, for $1\leq j\leq n$. Then the number of $n$-element independent sets of $\bar{F}$ depends only on $b_i$'s, $n,$ and $\sum_{i=1}^s{q_i}.$

\end{Corollary}
\begin{proof}
Without loss of generality, suppose $s=2$ and denote $\mcA_n=\mcA_n(A_2)$.
The number of $n$-element independent sets of $\bar{G}(\mcA_n(A_2),q)$ is: $$s_n(\bar{G}(\mcA_n(A_2),q))=\sum_{\substack{i,j\geq 0\\i+j=n}}\frac{\chi_{\mathcal{A}_{i}}(q)}{i!}\cdot\frac{\chi_{\mathcal{B}_{j}}(q-i)}{j!}.$$
Note that both $\mathcal{A}_n$ and $\mathcal{B}_n$ form exponential sequences of arrangements.  We have the following calculation:
\begin{align*}
s_n(\bar{F})&=\sum_{\substack{i_1,i_2,j_1,j_2\geq 0\\i_1+i_2+j_1+j_2=n}}\frac{\chi_{\mathcal{A}_{i_1}}(q_1)}{i_1!}\cdot\frac{\chi_{\mathcal{B}_{j_1}}(q_1-i_1)}{j_1!}\cdot\frac{\chi_{\mathcal{A}_{i_2}}(q_2)}{i_2!}\cdot\frac{\chi_{\mathcal{B}_{j_2}}(q_2-i_2)}{j_2!}\\
&=\sum_{\substack{i_1,i_2,j\geq 0\\i_1+i_2+j=n}}\left(\frac{\chi_{\mathcal{A}_{i_1}}(q_1)}{i_1!}\cdot\frac{\chi_{\mathcal{A}_{i_2}}(q_2)}{i_2!}\right)\cdot\frac{\chi_{\mathcal{B}_{j}}(q_1+q_2-i_1-i_2)}{j!}\\
&=\sum_{\substack{i,j\geq 0\\i+j=n}}\frac{\chi_{\mathcal{A}_{i}}(q_1+q_2)}{i!}\cdot\frac{\chi_{\mathcal{B}_{j}}(q_1+q_2-i)}{j!}\\
& =s_n(\bar{G}(\mcA_n(A_2),q)).
\end{align*}
This implies the number of independent sets of the disjoint union $\bar{F} $ only depends on $b_i$, $n$ and $k_1+k_2$. 
\end{proof}

After studying three special cases of the parameter set $A$, we now form a conclusion for general $A\subset\mathbb{Z}\times\mathbb{Z}$.

\begin{Proposition}
For a general parameter set $A\subset\mathbb{Z}\times\mathbb{Z}$, we have the arrangements $\mcA_n(A)$ and the corresponding graph $G(\mcA_n(A),q)$ given in Section 2. Let $H$ be the disjoint union of $G(\mcA_n(A),q_1)\cup\cdots\cup G(\mcA_n(A),q_s)$ for $q_i$ being powers of large enough prime numbers. Then the number of $n$-element independent sets of $H$ depends only on $n,s,A$ and $\sum_i{q_i},$. Furthermore, this number is independent of $s$ if and only if $\mcA_n(A)$ is not essential.
\end{Proposition}
\begin{proof}
For simplicity, denote $\mcA_n=\mcA_n(A)$. Similar to the proof of Theorem \ref{thm31}, we apply exponential formula and obtain
$$\sum^{\infty}_{n=0}\chi_{\mathcal{A}_n}(q_i)\frac{x^n}{n!}=\exp\left(\sum_{n= 1}^\infty\tilde{\chi}_{\mathcal{A}_n}(q_i)\frac{x^n}{n!}\right), \;\; i = 1,2,...,s$$
where $\tilde{\chi}_{\mathcal{A}_n}(q)=d_nq+c_n$.

Using Proposition \ref{24}, the independence polynomial of $H$ is as follows:
\begin{align*}
I_H(x)=\sum^{\infty}_{n=0}s_n{x^n}&=\exp\left(\sum_{n= 1}^\infty (\tilde{\chi}_{\mathcal{A}_n}(q_1)+\cdots +\tilde{\chi}_{\mathcal{A}_n}(q_s))\frac{x^n}{n!}\right)\\
&=\sum^{\infty}_{n=0}\chi_{\mathcal{A}_n}(\sum_i{q_i}+(s-1)\cdot\frac{c_n}{d_n})\frac{x^n}{n!}.
\end{align*}

Therefore, the number of $n$-element independent set in $H$ only depends on $n,s,A,\sum_i{q_i}$. It is independent of $s$ if and only if $c_n=0$. Plugging $q_i=0$, we get $\chi_{\mathcal{A}_{n}}(0)=0$ for $n\geq1$. This is equivalent to $\rank(\mathcal{A}_n) \leq n-1,$ i.e. $\mcA_n$ is not essential.
\end{proof}

\section{Relationship with the extended Catalan Arrangements}
In this section, we study the characteristic polynomial of $\chi_{\mcA_n(A_1)}$ again, but do not restrict to the multiplicatively independent parameters. We will relate it with the extended Catalan arrangements and compute it explicitly in several special cases.

\begin{Theorem}\label{thm51}
Let $\mcA_n(A_1)=\mathcal{B}_n\cup\{x_i=0\mid1\leq i\leq n\}\cup\{x_i=a_kx_j\mid 1\leq i\neq j\leq n,1\leq k\leq m\}$ be as before with all $a_i\in\mathbb{N}_{>1}$ and $\tilde{\mathcal{C}}_n$ be a deformation of the braid arrangement:
$$\tilde{\mathcal{C}}_n=\{x_i-x_j=b_k\mid 1\leq i\neq j\leq n,0\leq k\leq m\}$$
where $b_0=0$ and $b_k=\log a_k/\log a_1$ for $1\leq k\leq m$. Then $\chi_{\mathcal{A}_n(A_1)}(t)=\chi_{\tilde{\mathcal{C}}_n}(t-1).$
\end{Theorem}

\begin{proof}
For simplicity, denote $\mcA_n=\mcA_n(A_1)$. First consider the regions of $\mcA_n$ contained in different big regions separated by hyperplanes $\{x_i=0\mid 1\leq i\leq n\}$. In the orthants which are intersections of $i$ positive half-spaces and $n-i$ negative half-spaces, we change variables via $x_l=a^{y_l}_1$ when $x_l\geq 0$ and $x_l=-a^{y_l}_1$ when $x_l< 0$, then the number of regions of $\mcA_n$ in this big region is equal to $r(\tilde{\mathcal{C}}_{i})\cdot r(\tilde{\mathcal{C}}_{n-i})$.
Hence $$r(\mathcal{A}_n)=\sum^n_{i=0}\binom{n}{i}r(\tilde{\mathcal{C}}_{i})r(\tilde{\mathcal{C}}_{n-i}).$$
$$\sum^\infty_{n=0} r(\mathcal{A}_n)\frac{x^n}{n!}=\sum^\infty_{n=0} \sum^n_{i=0}\binom{n}{i}r(\tilde{\mathcal{C}}_{i})r(\tilde{\mathcal{C}}_{n-i})\frac{x^n}{n!}
=\left(\sum^\infty_{n=0} r(\tilde{\mathcal{C}}_n)\frac{x^n}{n!}\right)^2.$$
Using Theorem \ref{esa} and Corollary \ref{cor32}, we have 
\begin{align*}
\sum_{n\geq0}\chi_{\tilde{\mathcal{C}}_{n}}(t)\frac{(-x)^n}{n!}&=\left(\sum_{n\geq0} r(\tilde{\mathcal{C}}_{n})\frac{x^n}{n!}\right)^{-t}\\
\sum_{n\geq0} \chi_{\mathcal{A}_n}(t)\frac{(-x)^n}{n!}&=\left(\sum_{n\geq0} r(\mathcal{A}_n)\frac{x^n}{n!}\right)^{-(t-1)/2}.
\end{align*}
The above three equations leads to $\chi_{\mathcal{A}_n}(t)=\chi_{\tilde{\mathcal{C}}_{n}}(t-1).$
\end{proof}
\noindent

For applications, we give the following two results.

\begin{Proposition}
Let $\mcA_n=\mathcal{B}_n\cup\{x_i=0\mid1\leq i\leq n\}\cup\{x_i=a^kx_j\mid 1\leq i\neq j\leq n,1\leq k\leq m\}$ for $a\in\mathbb{N}_{>1}$. Then,
$\chi_{\mathcal{A}_n}(t)=(t-1)\prod^{n-1}_{j=1}(t-1-mn-j).$
\end{Proposition}
\begin{proof}
This is done directly by plugging ${a_1=a,\, a_2=a^2,\, ...,\, a_k=a^k}$ in Theorem \ref{thm51} and applying \cite[Thm 5.1]{athanasiadis}.
\end{proof}

\begin{Proposition}
Let $\mathcal{A}_n=\mathcal{B}_n\cup\{x_i=0\mid1\leq i\leq n\}\cup\{x_i =ax_j\mid 1\leq i<j\leq n\}$ for $a\in\mathbb{N}_{>1}$. Then $\chi_{\mathcal{A}_n}(t)=(t-1)(t-1-n)^{n-1}.$
\end{Proposition}
\begin{proof}
By plugging in $a_1=a$ and replacing all $i\neq j$ by $i< j$ in Theorem \ref{thm51}, and using the aharacteristic poluynomial of Shi arrangement ($\chi_{\mathcal{S}_n}(t)=t(t-n)^{n-1}$) in \cite[pp.64]{stanley1}, 
we have: $\chi_{\mathcal{A}_n}(t)=(t-1)(t-1-n)^{n-1}.$
\end{proof}

When $a_1,a_2,\ldots,a_m$ are multiplicatively independent positive integers, hyperplanes $x_i-x_j=\log a_k/\log a_1, 1\leq i\neq j\leq n,1\leq k\leq m$ are linearly independent, so Theorem \ref{thm51} alternatively proves Theorem \ref{thm31}. 
\subsection*{Acknowledgements} The authors would like to thank Prof. Richard P. Stanley for introducing this project to us and his useful suggestions. They would also thank MIT PRIMES program for supporting this research.

\bibliographystyle{unsrt}
\bibliography{mybib}

\begin{thebibliography}{1}

\bibitem{stanley1}
Richard~P. Stanley.
\newblock An introduction to hyperplane arrangements.
\newblock In {\em Lecture notes, IAS/Park City Mathematics Institute}, 2004.

\bibitem{whitney1932logical}
Hassler Whitney.
\newblock A logical expansion in mathematics.
\newblock {\em Bulletin of the American Mathematical Society}, 38(8):572--579,
  1932.

\bibitem{postnikov2000deformations}
Alexander Postnikov and Richard~P Stanley.
\newblock {Deformations of Coxeter hyperplane arrangements}.
\newblock {\em Journal of Combinatorial Theory, Series A}, 91(1-2):544--597,
  2000.

\bibitem{orlik2013arrangements}
Peter Orlik and Hiroaki Terao.
\newblock {\em Arrangements of hyperplanes}, volume 300.
\newblock Springer Science \& Business Media, 2013.

\bibitem{zaslavsky}
Thomas Zaslavsky.
\newblock {\em Facing up to Arrangements: Face-Count Formulas for Partitions of
  Space by Hyperplanes}, volume 154.
\newblock American Mathematical Soc., 1975.

\bibitem{garsia2002saga}
Adriano Garsia.
\newblock {\em The saga of reduced factorizations of elements of the symmetric
  group}, volume~29.
\newblock Universit{\'e} du Qu{\'e}bec {\`a} Montr{\'e}al, 2002.

\bibitem{stanley3}
Richard Stanley.
\newblock {\em Enumerative Combinatorics}, volume~2.
\newblock Cambridge University Press, Cambridge, 1999.

\bibitem{athanasiadis}
Christos~A Athanasiadis.
\newblock Characteristic polynomials of subspace arrangements and finite
  fields.
\newblock {\em Advances in mathematics}, 122(2):193--233, 1996.

\end{thebibliography}

\end{document}